\documentclass{compositio}
\usepackage{amsmath,amscd,amssymb,latexsym, amsfonts}
\usepackage{mathtools}

\theoremstyle{plain}

\newtheorem{theorem}{Theorem}[section]
\newtheorem{conjecture}[theorem]{Conjecture}
\newtheorem{proposition}[theorem]{Proposition}
\newtheorem{corollary}[theorem]{Corollary}

\newtheorem{lemma}[theorem]{Lemma}

\newcommand{\HT}[1]{\hat{\HH}{}^{#1}}

\theoremstyle{definition}
\newtheorem{definition}[theorem]{Definition}
\newtheorem{remark}[theorem]{Remark}

\DeclareMathOperator{\Gal}{Gal}

\DeclareMathOperator{\HH}{H}

\DeclareMathOperator{\Nm}{Nm}
\DeclareMathOperator{\Hom}{Hom}
\DeclareMathOperator{\Spec}{Spec}
\DeclareMathOperator{\Res}{Res}
\DeclareMathOperator{\Fr}{Fr}
\DeclareMathOperator{\Ind}{Ind}

\DeclareMathOperator{\Z}{Z}

\DeclareMathOperator{\GL}{GL}
\DeclareMathOperator{\PGL}{PGL}
\DeclareMathOperator{\SL}{SL}

\newcommand{\TT}{\mathcal{T}}
\newcommand{\C}{\mathcal{C}}
\newcommand{\CC}{\mathbb{C}}
\newcommand{\CCx}{\mathbb{C}^\times}
\newcommand{\OK}{\mathcal{O}_K}
\newcommand{\OKn}{\mathcal{O}_{K_n}}
\newcommand{\PK}{\mathcal{P}_K}
\newcommand{\PL}{\mathcal{P}_L}
\newcommand{\OL}{\mathcal{O}_L}
\newcommand{\ZZ}{\mathbb{Z}}
\newcommand{\QQ}{\mathbb{Q}}
\newcommand{\Gm}{\mathbb{G}_m}
\newcommand{\Kx}{K^\times}
\newcommand{\Lx}{L^\times}
\newcommand{\Fq}{\mathbb{F}_q}
\newcommand{\Fqb}{\bar{\mathbb{F}}_q}

\newcommand{\Weil}{\mathcal{W}}

\newcommand{\Lpack}{\mathcal{L}}
\newcommand{\Pmin}{P_G^{\min}}
\newcommand{\bmu}{\boldsymbol\mu}
\newcommand{\mumin}{\bmu^{\min}}

\newcommand{\st}{\ensuremath{\ \ \ \vert\ }}

\newcommand{\invlim}[1]{\varprojlim_{#1}}
\newcommand{\Normalizer}[2]{\operatorname{N}_{#2}(#1)}

\newcommand{\Thadm}{T^*_{\operatorname{adm}}}
\newcommand{\Thinadm}{T^*_{\operatorname{in}}}
\newcommand{\hatT}{T^*}

\begin{document}
\title[Rectifiers and the local Langlands Correspondence]{Rectifiers and the local Langlands Correspondence: the unramified case}
\author{Moshe Adrian}
\email{madrian@math.utah.edu}
\address{Department of Mathematics, University of Utah, Salt Lake City, UT 84112, U.S.A.}
\author{David Roe}
\email{roed.math@gmail.com}
\address{Department of Mathematics, University of Calgary, Calgary, AB T2N 1N4, Canada}
\thanks{The second author was supported by the Pacific Institute for the Mathematical Sciences}

\begin{abstract}

We generalize the rectifier of Bushnell and Henniart,
which occurs in the local Langlands correspondence for
$\GL_{n}(K)$, to certain Langlands parameters for
unramified connected reductive groups.

\end{abstract}

\maketitle

\section{Introduction} \label{section:intro}

Let $G$ be a connected reductive group defined over a $p$-adic field $K$.
The local Langlands conjecture predicts the existence of a finite to one map
from the set of isomorphism classes of irreducible admissible representations
of $G(K)$ to the set of Langlands parameters for $G(K)$.

There has been a significant amount of progress in recent years
focusing on supercuspidal representations of $G(K)$.  Bushnell-Henniart \cite{bushnell-henniart:10a},
DeBacker-Reeder \cite{reeder-debacker:09a}, Kaletha \cite{kaletha:13a} and Reeder \cite{reeder:08a}
approach the task of constructing $L$-packets by first attaching
a character of an elliptic torus to a Langlands parameter and
then associating a collection of supercuspidal representations to this character.
Their constructions all use the local Langlands correspondence for tori in some way,
but the image of the Langlands parameter $\Weil_K \rightarrow {}^L G$ is not
necessarily contained within the $L$-group of a maximal torus.  The different authors
remedy this situation in various ways.

For an example where the image does not land in an $L$-group of a maximal torus,
consider $G = \PGL_{2}(K)$.  Suppose that $p \neq 2$ and that
$\varphi : \Weil_K \rightarrow \SL_{2}(\CC)$ is an
irreducible representation.  Then there is a tamely ramified quadratic
extension $L/K$ and a character
$\chi$ of $\Lx$, trivial on the norms $\Nm_{L/K}(\Lx)$, so that
$\varphi = \Ind_{\Weil_L}^{\Weil_K}(\chi)$.  The image of $\varphi$
is contained in the normalizer of the dual torus $\hat{T}$, a non-split
extension of $\Gal(L/K)$ by $\hat{T}$, but not in the $L$-group of any torus.

The group $\Lx / \Nm_{L/K}(\Lx)$ appears as a cover of the elliptic torus $L^1$ of norm $1$ elements in $L$:
\begin{align*}
1 \rightarrow \mathbb{Z} / 2 \mathbb{Z} \rightarrow \Lx / \Nm_{L/K}(\Lx) &\rightarrow L^1 \rightarrow 1 \\
x \Nm_{L/K}(L^{\times}) &\mapsto x / \sigma(x);
\end{align*}
here $\sigma$ generates $\Gal(L/K)$.  In particular, the Langlands parameter
$\varphi$ naturally provides a character $\chi$, not of the elliptic torus
$L^1 \subset \PGL_2(K)$,
but of the two-fold cover $\Lx / \Nm_{L/K}(\Lx)$.  We can obtain a character of $L^1$
by twisting $\chi$ by a genuine character of
$\Lx / \Nm_{L/K}(\Lx)$.  The twist giving the correct supercuspidal representation of $\PGL_2(K)$
is precisely what appears in Bushnell and Henniart \cite{bushnell-henniart:06a, bushnell-henniart:10a}.

In this paper, we generalize Bushnell and Henniart's
rectifier to groups other than $\GL_n(K)$.  In particular, we define rectifiers for
unramified minisotropic tori $T$ in connected reductive groups $G$.
Benedict Gross' recent construction of groups of type L provides a framework for
us to define rectifiers and admissible pairs in the general setting.
We show in Theorem \ref{thm:unique_semisimple} that rectifiers for semisimple $G$ exist
and are unique up to equivalence. In the setting of depth zero
supercuspidal representations of $\GL_{n}(K)$, Theorem \ref{thm:bh_agreement}
gives the compatibility of our rectifier with that of Bushnell and Henniart.
We note that there is an obstruction to proving compatibility in the positive depth case.
In the depth zero case, Deligne-Lusztig representations provide a canonical way of
constructing supercuspidal representations. However, in positive depth there are many:
Adler \cite{adler:98a}, Howe \cite{howe:77a}, Bushnell-Henniart \cite{bushnell-henniart:10a},
Bushnell-Kutzko \cite{bushnell-kutzko:AdmissibleDual}, and Yu \cite{yu:03a}.
In the positive depth setting, the rectifier will depend on the methods
used to construct representations from the character of $T(K)$, and
our rectifier indeed differs from that of Bushnell and Henniart in positive depth.

Bushnell and Henniart motivate their rectifier as follows.
Suppose that $\varphi$ is an \emph{essentially tame} supercuspidal Langlands parameter for
$\GL_n(K)$.  The local Langlands correspondence for tori then yields a degree $n$ extension
$L/K$ and a character $\xi$ of $L^{\times}$.  We now fix a construction $\chi \mapsto \pi_{\chi}$
of supercuspidal representations of $\GL_n(K)$ from \emph{admissible} characters of $L^{\times}$.
Then the rectifier of $\xi$ is a character $\mu_{\xi}$ of $L^{\times}$ such that
$\varphi \mapsto \pi_{\xi \cdot \mu_{\xi}}$ is the local Langlands correspondence for $\GL_n(K)$.

We generalize their notion of rectifier to unramified connected reductive groups $G$.
Suppose that $\varphi : \Weil_K \rightarrow {}^L G$ is a supercuspidal Langlands parameter
that factors through the normalizer of a maximal torus.
The local Langlands for tori again provides a canonical way to proceed.  Assuming a mild
cohomological condition, one
obtains from $\varphi$ a character $\xi$ of a cover of an elliptic torus.
After fixing an association $\chi \mapsto \Lpack(\chi)$ of supercuspidal $L$-packets of $G$
to \emph{admissible} characters of this cover, the rectifier of $\xi$ is a character
$\mu_{\xi}$ of the cover such that $\varphi \mapsto \Lpack(\xi \cdot \mu_{\xi})$
is the local Langlands correspondence for $G(K)$.

We now present an outline of the paper.  In \S\ref{section:BH_recall} we recall
the notion of rectifier due to Bushnell and Henniart and describe
the rectifier in the setting that we will need.  In \S\ref{section:padic_tori}
we present some results about Tate cohomology of $p$-adic tori that will be used
in the rest of the paper.  In \S\ref{section:groups_of_type_L} we review
the theory of ``groups of type L.''  In
\S\ref{section:gross_debacker_reeder} we describe the relationship between the
construction of Gross, via groups of type L, and the constructions of
DeBacker-Reeder and Reeder.  In \S\ref{Q_T} we study how translation by a character affects
the association $\chi \mapsto \Lpack(\chi)$.
In \S\ref{section:general_rectifiers} we
introduce our notion of rectifier and prove our main result, Theorem \ref{thm:unique_semisimple}.
Finally, in \S\ref{section:BH_compat} we show that our rectifier is compatible with
the rectifier of Bushnell and Henniart in the setting of depth zero
supercuspidal representations of $\GL_n(K)$.

\subsection*{Acknowledgements}

This paper has benefited from conversations with Jeffrey Adams, Jeffrey Adler, Colin Bushnell, Andrew Fiori, Guy Henniart, Gordan Savin, and Geo Kam-Fai Tam.  We thank them all.

\section{Notation and Preliminaries} \label{section:notation}

Throughout, $K$ will denote a nonarchimedean local field of
characteristic zero, $\OK$ its ring of integers, $k$ its residue field,
$\PK$ the maximal ideal in $\OK$ and $\varpi$ a fixed uniformizer.
Write $K_n$ for the unramified extension of $K$ of degree $n$, $k_n$ for
the degree $n$ extension of $k$,
and set $\Gamma_n = \Gal(K_n/K) = \Gal(k_n/k)$.

A geometric Frobenius is an element of $\Gal(\bar{K}/K)$
inducing the automorphism $x \mapsto x^{1/p}$ of $\bar{k}$.  Under the
Artin reciprocity map of local class field theory the choice of $\varpi$
determines a geometric Frobenius $\Fr$ \cite[\S 2]{serre:LocalClassFieldThy}.

If $\chi : K^{\times} \rightarrow \mathbb{C}^{\times}$ is a character, we define
the \emph{depth} of $\chi$ to be the smallest integer $r$ such that
$\chi|_{1 + \PK^{r+1}} \equiv 1$ and
$\chi|_{1 + \PK^{r}} \not\equiv 1$.

If $T$ is a torus defined over $K$ we write $X^*(T)$
for the character lattice $\Hom_{\bar{K}}(T, \Gm)$ and $X_*(T)$ for the
cocharacter lattice $\Hom_{\bar{K}}(\Gm, T)$ \cite[\S 16.2]{humphreys:LinAlgGrps}.
$T$ will split over an extension
$L$ of $K$ if and only if $\Gal(\bar{K}/L)$ acts trivially on $X^*(T)$.
We may thus define \emph{the} splitting field $L$ of $T$ as the
minimal extension of $K$ splitting $T$; note that $L$ is necessarily
Galois over $K$.  Write $\Gamma$ for $\Gal(L/K)$; $X_*(T)$, $X^*(T)$ and $T(L)$
are all $\Gamma$-modules.

Suppose now that $T \subset G$ for a connected reductive group $G$ over $K$.
We will write $\hat{T} \subset \hat{G}$ for the dual torus in the complex dual group of $G$ \cite[\S I.2]{borel:79a}.
Let $N$ be the normalizer $\Normalizer{T}{G}$ of $T$ in $G$ and define $W = N/T$;
set $\hat{N} = \Normalizer{\hat{T}}{\hat{G}}$ and
$\hat{W} = \hat{N}/\hat{T}$.  The identification of $X^*(T)$ and $X_*(\hat{T})$
yields a canonical anti-isomorphism between $W$ and $\hat{W}$.
Note that $W$ is a scheme over $K$; in general $W(K) \ne N(K) / T(K)$.

Write $\Nm$ for the norm map
\begin{align*}
T(L) &\rightarrow T(K) \\
t &\mapsto \prod_{\sigma \in \Gamma} \sigma(t)
\end{align*}
and for its restriction to $X_*(T)$.

The following theorem, due to Lang \cite{lang:56a}, underpins the facts in
\S\ref{section:padic_tori} on tori over $p$-adic fields.
Let $H$ be a commutative connected algebraic group over a
finite field $k$, and suppose $H$ splits over $k_n$.  Denote by $\HT{i}$ the $i^{\mathrm{th}}$
Tate cohomology group.

\begin{theorem} \label{thm:lang}
$\HT{i}(\Gamma_n, H(k_n)) = 0$ for all $i$.
\end{theorem}
\begin{proof}
Since $\Gamma_n$ is cyclic,
$\HT{i}(\Gamma_n, H(k_n)) \cong \HT{i+2}(\Gamma_n, H(k_n))$ \cite[Thm. 5]{atiyah-wall:CohomologyGrps},
so it suffices to prove the result for $i=1$ and $i=2$, which is done
by Serre \cite[\S VI.6]{serre:AlgGrpsClassFields}.
\end{proof}

\section{Rectifier for $\GL_{n}(K)$} \label{section:BH_recall}

In this section we recall the rectifier of Bushnell and Henniart and their construction of the
essentially tame local Langlands correspondence for $\GL_{n}(K)$.
An irreducible smooth representation of the Weil group $\Weil_K$ of $K$ is
called \emph{essentially tame} if its restriction to wild inertia is a
sum of characters.
\begin{definition}\label{admissiblepairhowe}
Let $L/K$ be an extension of degree $n$, with $n$ coprime to $p$.  A character
$\xi$ of $L^\times$ is \emph{admissible} if
\begin{enumerate}
\item $\xi$ doesn't come via the norm from a subfield of $L$ containing $K$,
\item If $\xi|_{1 + \PL}$ comes via the norm from a subfield $L \supset M \supset K$, then
$L/M$ is unramified.
\end{enumerate}
\end{definition}
There is a natural bijection
$\varphi_{\xi} \leftrightarrow (L/K, \xi)$ between irreducible smooth essentially tame
$\varphi_{\xi} : \Weil_K \rightarrow \GL_{n}(\mathbb{C})$ and
\emph{admissible pairs} $(L/K, \xi)$.
Bushnell and Henniart
construct a map (see \cite{bushnell-henniart:10a})
\begin{equation*}
\left\{
\begin{array}{cc}
\mathrm{isomorphism \ classes \ of} \\
\mathrm{admissible \ pairs}
\end{array}
\right\} \rightarrow \left\{
\begin{array}{cc}
\mathrm{supercuspidal \ representations} \\
\mathrm{of} \ \GL_{n}(K)
\end{array} \right\}
\end{equation*}
$$\hspace{-.5in} (L/K, \xi) \mapsto \pi_{\xi}$$
However, the map $$\varphi_{\xi} \mapsto \pi_{\xi}$$
is not the local Langlands
correspondence because $\pi_{\xi}$ has the wrong central character.
Instead, the local Langlands correspondence is given by
\begin{equation}\label{llcgln}
\varphi_{\xi} \mapsto \pi_{\xi \cdot {}_K \mu_{\xi}} \tag{$\star$}
\end{equation}
for some subtle finite order
character ${}_K \mu_{\xi}$ of $\Lx$.  Since we will not be changing $K$
in this paper we will write $\mu_\xi$ for ${}_K \mu_{\xi}$.

The relation $\star$ does not determine $\mu_{\xi}$ uniquely.  As pointed out
in \cite{bushnell-henniart:10a}, the obstruction to uniqueness revolves around the
group $\GL_2(\mathbb{F}_3)$.  Bushnell and Henniart therefore make the following definition \cite[Def. 1]{bushnell-henniart:10a}.

\begin{definition}\label{rectifierbushnellhenniart}
Let $L/K$ be a finite, tamely ramified field extension of degree $n$.  A \emph{rectifier}
for $L/K$ is a function
$$\bmu : (L/K, \xi) \mapsto \mu_{\xi}$$
which attaches to each admissible pair $(L/K, \xi)$ a character $\mu_{\xi}$ of $L^{\times}$
satisfying the following conditions:
\begin{enumerate}
\item The character $\mu_{\xi}$ is tamely ramified.
\item Writing $\xi' = \xi \cdot \mu_{\xi}$, the pair $(L/K, \xi')$ is admissible and
$\varphi_{\xi} \mapsto \pi_{\xi \cdot \mu_{\xi}}$ is the local Langlands correspondence
for $\GL_n(K)$.
\item If $(L/K, \xi_i), i = 1,2$, are admissible pairs such that $\xi_1^{-1} \xi_2$ is
tamely ramified, then $ \mu_{\xi_1} =  \mu_{\xi_2}$.
\end{enumerate}
\end{definition}

Bushnell and Henniart then prove \cite[Thm. A]{bushnell-henniart:10a}:

\begin{theorem}
Any finite, tamely ramified, field extension $L/K$ admits a unique rectifier
$\bmu : (L/K, \xi) \mapsto \mu_{\xi}$.
\end{theorem}

Both the description of and the
intuition behind the rectifiers $\bmu$ have been
studied (see \cite{bushnell-henniart:10a}, \cite{tam:12a}, \cite{adrian:13a}).  In order to
generalize rectifiers to groups other than $\GL_n(K)$ we
will will need a description of the characters $\mu_{\xi}$ in certain cases.
Let us recall some notions from \cite[\S8]{bushnell-henniart:10a}.

Let $(L/K,\xi)$ be an admissible pair and let $i \in \ZZ_{\ge 0}$.
There is a minimal sub-extension $L_i/K$ of $L/K$ such that
$\xi|_{1 + \PL^{i+1}}$ factors through the norm $\Nm_{L/L_i}$.  We say that $i \in \mathbb{Z}$
is a \emph{jump of} $\xi$ \emph{over} $K$ if $i \geq 1$ and $L_{i-1} \neq L_i$.

\begin{proposition}\label{prop:BH_result1}
  Suppose that $(L/K, \xi)$ is an admissible pair, where $L/K$
  is unramified and $\xi$ has depth $0$.
  Then $\mu_{\xi}$ is unramified and
  $\mu_{\xi}(\varpi) = (-1)^{n-1}$.
\end{proposition}

\begin{proof}
It is clear that the set of jumps of $\xi$ over $K$ is empty.
Therefore, by \cite[Proposition 21]{bushnell-henniart:10a}, we have the result.
\end{proof}

\section{Tori over $p$-adic fields} \label{section:padic_tori}

Let $T$ be a torus defined over $K$ with splitting field $L$, let $K_n$ be the maximal
unramified subextension $L/K$ and set $I = \Gal(L/K_n)$.
Let $\TT$ be the N\'eron model of $T$, a canonical model of $T$
over $\OK$ \cite[Ch. 10]{bosch-lutkebohmert-reynaud:NeronModels}.
As a consequence of the N\'eron mapping
property, we may identify $\TT(\OK)$ with $T(K)$.  The connected
component of the identity, $\TT^\circ$, cuts out a subgroup
$T(K)_0 = \TT^\circ(\OK)$ of $T(K)$; we also write $T(K_n)_0$ for
$\TT^\circ(\OKn)$.

In fact, this subgroup of $T(K)$ is the first in a decreasing filtration.
Moy and Prasad \cite{moy-prasad:96a}
define one such filtration by
embedding $T$ into an induced torus and defining the filtration of
$\Res_{L/K} \Gm$ in terms of the valuation on $L$.  Yu \cite[\S 5]{yu:03a}
describes a different filtration, agreeing with that of Moy and Prasad
in the case of tame tori but with nicer features in the presence of wild
ramification.  Let $\{\TT_r\}_{r \ge 0}$ be the integral models of $T$ defined in Yu's
minimal congruent filtration and let $\{T(K)_r\}_{r \ge 0}$ and
$\{T(K_n)_r\}_{r \ge 0}$ be the corresponding filtrations of $T(K)$ and
$T(K_n)$.

Let $\C$ be the scheme of
connected components of $\TT$,
which we may identify with the
components of $\TT \times \Spec(k)$ since $T = \TT \times \Spec(K)$
is connected.  The structure of $\C$ is described by Xarles:

\begin{proposition}[{\cite[Cor. 2.12]{xarles:93a}}]
There is an exact sequence of $\Gamma_n$-modules
$$0 \rightarrow \Hom_{\ZZ}(\HH^1(I, X^*(T)), \QQ/\ZZ) \rightarrow
\C \rightarrow \Hom_{\ZZ}(X^*(T)^I, \ZZ) \rightarrow 0.$$
\end{proposition}

\begin{corollary}[{\cite[Thm. 1.1]{xarles:93a}}] \label{cor:unram_components}
If $T$ is unramified, then $\C \cong X_*(T)$.
\end{corollary}

Using our filtration of $T(K_n)$, we may relate the cohomology of $T(K_n)$
with that of $\C$.

\begin{proposition}\label{prop:T0_cohom_triv}
$\HT{i}(\Gamma_n, T(K_n)_0) = 0$ for all $i$.
\end{proposition}
\begin{proof}
Note that
$$T(K_n)_0 = \invlim{r} T(K_n)_0 / T(K_n)_r.$$
So by a result of Serre \cite[Lem. 3]{serre:LocalClassFieldThy}, it suffices to prove that
\\ $\HT{i}(\Gamma_n, T(K_n)_r / T(K_n)_{r+}) = 0$ for all $i$.  But $T(K_n)_r / T(K_n)_{r+}$
is connected \cite[Prop. 5.2]{yu:03a} and thus has trivial cohomology by
Theorem \ref{thm:lang}.
\end{proof}

\begin{corollary}
$\HT{i}(\Gamma_n, T(K_n)) \cong \HT{i}(\Gamma_n, \C)$.
\end{corollary}

\begin{proof}
This follows from the long exact sequence in cohomology associated to the sequence
$$0 \rightarrow \TT^0 \rightarrow \TT \rightarrow \C \rightarrow 0.$$
\end{proof}

Suppose now that $T$ is unramified with splitting field $L = K_n$.

\begin{corollary} \label{cor:cohom_tori}
If $T$ is unramified, then $\HT{i}(\Gamma_n, T(L)) \cong \HT{i}(\Gamma_n, X_*(T))$
for all $i$.
\end{corollary}

\begin{proof}
This follows from the previous corollary together with Corollary \ref{cor:unram_components}.
\end{proof}

\begin{corollary}\label{cor:vanishing_H0}
If $T$ is unramified and anisotropic, then $\HT{0}(\Gamma_n, T(L)) = 0$.
\end{corollary}

\begin{proof}
Since $T$ is anisotropic, $X_*(T)^{\Gamma_n} = 0$, giving $\HT{0}(\Gamma_n, T(L)) = 0$ by Corollary
\ref{cor:cohom_tori}.
\end{proof}

For unramified $T$ the jumps in the filtration on $T(K)$ and $T(L)$ occur at integers, and we write
\begin{align*}
T(\OK) &= T(K)_0, \\
T(\OL) &= T(L)_0, \\
T(\PK^r) &= T(K)_r\qquad \mbox{for $r > 0$}, \\
T(\PL^r) &= T(L)_r\qquad\,\,\mbox{for $r > 0$}.
\end{align*}

\section{Groups of type L} \label{section:groups_of_type_L}

We now review the theory of groups of type L due to Benedict
Gross.  For a torus $T$ over $K$ recall that the dual torus $\hat{T}$ is equipped with
an action of $\Gamma$.

\begin{definition}
A \emph{group of type L} is a group extension of $\Gamma$ by $\hat{T}$.
\end{definition}

For such a group $D$ we have by definition an exact sequence
$$1 \rightarrow \hat{T} \rightarrow D \rightarrow \Gamma \rightarrow 1.$$

We now describe how we can naturally attach a character of the coinvariants
$T(L)_{\Gamma}$ to a Langlands parameter
$$\varphi : \Weil_K \rightarrow D$$
with values in a group of type L.
Restricting $\varphi$ to $\Weil_L$ we get a homomorphism
$$\varphi|_{\Weil_L} : \Weil_L \rightarrow \hat{T},$$
and by the Langlands correspondence for tori a character
$\xi_{\varphi} : T(L) \rightarrow \CCx$.  Since $\varphi|_{\Weil_L}$ extends
to $\varphi$ we have that
$$\xi_{\varphi}(\sigma(t)) = \xi_{\varphi}(t)\ \mbox{for all $\sigma \in \Gamma$.}$$
Thus $\xi_{\varphi}$ is trivial on the augmentation ideal $I_{\Gamma}(T(L))$
and descends to $$\xi_{\varphi} : T(L)_\Gamma \rightarrow \CCx.$$

Invariants
and coinvariants are related by the norm map
in the Tate cohomology sequence
$$1 \rightarrow \HT{-1}(\Gamma,T(L)) \rightarrow T(L)_{\Gamma} \xrightarrow{\Nm} T(K)
  = T(L)^{\Gamma} \rightarrow \HT{0}(\Gamma,T(L)) \rightarrow 1.$$
We will assume in \S \ref{section:general_rectifiers} that $\HT{0}(\Gamma,T(L)) = 0$, in which case
$\xi_\varphi$ is a character of a cover of $T(K)$.

We will need the following structural result about Langlands
parameters mapping to groups of type L for the proof of
Proposition \ref{prop:existenceofrectifier}.  Suppose now that $L/K$ is unramified and that
$\varphi$ and $\varphi'$ are two Langlands parameters
with $\varphi'(\Fr) \varphi(\Fr)^{-1} \in \hat{T}$.
Let $\xi$ and $\xi'$ be the associated characters of $T(L)_{\Gamma}$.

\begin{lemma} \label{lem:toral_modification}
$\xi$ and $\xi'$ have the same restriction to $\HT{-1}(\Gamma, T(L))$.
\end{lemma}

\begin{proof}
It suffices to prove that $\xi' \cdot \xi^{-1}$ vanishes on
$\ker(\Nm : T(L) \rightarrow T(K))$.  Define $g \in D$ and $t \in \hat{T}$ by
$\varphi(\Fr) = g$, $\varphi'(\Fr) = tg$.  Then
\begin{align*}
\varphi'(\Fr^n) \varphi(\Fr^n)^{-1} &= (tg)^n g^{-n} \\
&= \prod_{i=0}^{n-1} g^i t g^{-i} \\
&= \prod_{i=0}^{n-1} \Fr^i(t)
\end{align*}
since $g$ projects to $\Fr \in \Gamma$.  Define $\varphi_i \colon \Weil_L \rightarrow \hat{T}$
by $\varphi_i(z) = 1$ for $z \in I_L$ and
$\varphi_i(\Fr^n) = \Fr^i(t)$; let $\xi_i$ be the associated character
of $T(L)$.  By \cite[Lem. 4.3.1]{reeder-debacker:09a}, $\xi_i = \xi_0 \circ \Fr^i.$
Suppose that $x \in T(L)$ with $\Nm(x) = 1$.  Then
\begin{align*}
\xi'(x) \xi(x)^{-1} &= \prod_{i=0}^{n-1} \xi_i(x) \\
&= \xi_0 \left(\prod_{i=0}^{n-1}\Fr^i(x)\right) \\
&= 1.\\
\end{align*}
\end{proof}

We will also need the following lemma in order to define our notion of admissible pair
in \S\ref{section:general_rectifiers}.

\begin{lemma} \label{lem:weyl_groups}
Let $G$ be a connected reductive $K$-group and let $T$ be a maximal
$K$-torus of $G$.
\begin{enumerate}
\item $\Normalizer{T(K)}{G(L)} / T(L) \cong W(K)$.
\item The standard action of $\Normalizer{T(L)}{G(L)} / T(L)$ on $T(L)$ determines
actions of $\Normalizer{T(L)}{G(L)}^\Gamma / T(K)$ and $W(K)$
on $T(L)$ which factor naturally to actions on $T(L)_\Gamma$.
\end{enumerate}
\end{lemma}

\begin{proof}
See \cite[Lem. 9.1]{adrian-lansky:ppa}.
\end{proof}

\section{The relationship between the Gross construction and the DeBacker--Reeder and Reeder construction}
\label{section:gross_debacker_reeder}

Let $\varphi : \Weil_K \rightarrow {}^L G$ be a regular semisimple elliptic Langlands
parameter for an unramified connected reductive group $G$
(see \cite{reeder-debacker:09a} and \cite{reeder:08a}).
Here, ${}^L G = \langle \hat{\theta} \rangle \ltimes \hat{G}$,
where $\hat{\theta}$ is the dual Frobenius automorphism on $\hat{G}$
(see \cite[\S 3]{reeder-debacker:09a}).
Note that $\varphi$ has image in a group of
type L.  Let $L,K,T,\hat{T}, \Gamma$ and $\xi_{\varphi}$ be as in
\S\ref{section:groups_of_type_L} and recall that we have assumed that $L/K$ is unramified.
Then $\varphi(I_K) \subset \hat{T}$ and
$\varphi(\Fr) = \hat{\theta} f$ for some $f \in \hat{N}$.  Let $\hat{w}$
be the image of $f$ in $\hat{W}$.
DeBacker--Reeder \cite{reeder-debacker:09a} and Reeder \cite{reeder:08a}
associate a character $\chi_{\varphi}$ of $T(K)$ to $\varphi$.

We now recall the definition of
the Tits group and some of its properties.  Choose a set $\{ X_{\alpha} \}$ of root vectors
indexed by the set of simple roots of $\hat{T}$ in $\hat{B}$; $(\hat{T}, \hat{B}, \{X_{\alpha} \})$
is a pinning as in \cite[\S 3.1]{reeder:09a}.
For each simple root $\alpha$, define $\phi_{\alpha} : \SL_2 \rightarrow \hat{G}$
by
\begin{align*}
\phi_{\alpha}\begin{pmatrix}z & 0 \\ 0 & z^{-1}\end{pmatrix} &= \alpha^{\vee}(z) \\
d \phi_{\alpha}\begin{pmatrix}0 & 1 \\ 0 & 0\end{pmatrix} &= X_{\alpha}.
\end{align*}
Let $\sigma_{\alpha} = \phi_{\alpha}\begin{pmatrix}0 & 1 \\ -1 & 0\end{pmatrix}$.  

\begin{definition}
  The Tits group $\widetilde{W}$ is the subgroup of $\hat{N}$
  generated by $\{\sigma_{\alpha} \}$ for simple roots $\alpha$.
\end{definition}

For each simple root $\alpha$, let $m_{\alpha} = \sigma_{\alpha}^2 = \alpha^{\vee}(-1)$ and
let $\hat{T}_2$ be the subgroup of $\hat{T}$ generated by the $m_{\alpha}$.

\begin{theorem}{(\cite{tits:66a})}
\begin{enumerate}

\item The kernel of the natural map $\widetilde{W} \rightarrow \hat{W}$
  is $\hat{T}_2$,
\item The elements $\sigma_{\alpha}$ satisfy the braid relations,
\item There is a canonical lifting of $\hat{W}$ to a subset of
  $\widetilde{W}$: take a reduced expression $w = s_{\alpha_1} \cdots s_{\alpha_n}$,
  and let $\tilde{w} = \sigma_{\alpha_1} ... \sigma_{\alpha_n}$.
\end{enumerate}
\end{theorem}

We remark that the lifting $\hat{W} \rightarrow \widetilde{W}$ is not necessarily a homomorphism,
as shown by the example of $\SL_2$.

\begin{definition} \label{def:phiu}
Given $\hat{u} \in \hat{W}$, let $\tilde{u}$ be its canonical lift to $\widetilde{W}$.
We define a homomorphism $\varphi_{\hat{u}} : \Weil_K \rightarrow {}^L G$ by
\begin{enumerate}
\item $\varphi_{\hat{u}}|_{I_K} \equiv 1$,
\item $\varphi_{\hat{u}}(\Fr) = \hat{\theta} \tilde{u}$.
\end{enumerate}
\end{definition}

By
\S\ref{section:groups_of_type_L}, $\varphi$ and $\varphi_{\hat{w}}$ give rise to characters
$\xi_{\varphi}$ and $\xi_{\varphi_{\hat{w}}}$ of $T(L)_{\Gamma}$ respectively.

\begin{lemma} \label{lem:GDR_compat}
$\xi_{\varphi}$ and $\chi_{\varphi} \circ \Nm$ have the same restriction to $T(\OL)_{\Gamma}$.
\end{lemma}

\begin{proof}
We have the exact sequence
$$1 \rightarrow \HT{-1}(\Gamma, T(L)) \rightarrow T(L)_{\Gamma} \rightarrow T(K)
  \rightarrow \HT{0}(\Gamma, T(L)) \rightarrow 1.$$
Recall that the character $\xi_{\varphi}$ is associated to $\varphi$ by
the local Langlands correspondence for tori (see \S\ref{section:groups_of_type_L}).
Note that the above exact sequence restricts to an exact sequence
$$1 \rightarrow \HT{-1}(\Gamma, T(\OL)) \rightarrow T(\OL)_{\Gamma}
  \rightarrow T(\OK) \rightarrow \HT{0}(\Gamma, T(\OL)) \rightarrow 1.$$
Moreover, by Proposition \ref{prop:T0_cohom_triv}, we have
$\HT{-1}(\Gamma, T(\OL)) = \HT{0}(\Gamma, T(\OL)) = 1$.
Therefore, the map
$$T(\OL)_{\Gamma} \xrightarrow{\Nm} T(\OK)$$
is an isomorphism, so
$\xi_{\varphi}|_{T(\OL)_{\Gamma}}$ factors to a character of
$T(\OK)$ via this isomorphism.  But this is exactly how the character
$\chi_{\varphi}|_{T(\OK)}$ is constructed in \cite{reeder-debacker:09a} and \cite{reeder:08a}.
\end{proof}

The following proposition relates the character $\xi_\varphi$ defined through groups of type $L$ to the character $\chi_\varphi$ constructed by DeBacker-Reeder and Reeder.

\begin{proposition}\label{prop:existenceofrectifier}
If $G$ is semisimple, then $\chi_{\varphi} \circ \Nm = \xi_{\varphi} \otimes \xi_{\varphi_{\hat{w}}}^{-1}$.
\end{proposition}

\begin{proof}
Since $G$ is semisimple, $T(K)$ is compact.  In particular,
$\HT{0}(\Gamma, T(L)) = 0$ by Corollary \ref{cor:vanishing_H0},
so we have the following exact sequence:
$$1 \rightarrow \HT{-1}(\Gamma, T(L)) \rightarrow T(L)_{\Gamma} \rightarrow T(K) \rightarrow 1.$$
Note that $T(K) = T(\OK)$ and thus
$T(\OL)_{\Gamma}$ surjects onto $T(K)$ via the norm map
$\Nm$.  Therefore $\HT{-1}(\Gamma,T(L))$ and
$T(\OL)_{\Gamma}$ together generate $T(L)_{\Gamma}$.  It thus suffices to check that
$\xi_{\varphi} \otimes \xi_{\varphi_{\hat{w}}}^{-1} = \chi_{\varphi} \circ \Nm$
on each of these two subgroups.

Since $\varphi_{\hat{w}}|_{I_K} \equiv 1$, $\xi_{\varphi_{\hat{w}}}$ is trivial on
$T(\OL)_{\Gamma}$ so Lemma
\ref{lem:GDR_compat} implies equality on $T(\OL)_{\Gamma}$.
Equality on $\HT{-1}(\Gamma,T(L))$ is Lemma \ref{lem:toral_modification}.
\end{proof}

We note that for semisimple $G$ we may replace $\tilde{w}$ by another
lift $w'$ of $\hat{w}$ to $\hat{N}$ in the definition of $\varphi_{\hat{w}}$.
In fact, if we define $\varphi'$ by
\begin{align*}
\varphi'|_{I_K} &\equiv 1 \\
\varphi'(\Fr) &= w'
\end{align*}
then Lemma \ref{lem:toral_modification} implies $\xi_{\varphi_{\hat{w}}} = \xi_{\varphi'}$.
We will justify the Tits group lift $\tilde{w}$ in \S\ref{section:BH_compat} for $\GL_n(K)$.

\section{$L$-packets fixed under translation by a character}\label{Q_T}

The general definition of rectifier is complicated by the fact that different
characters of a torus can yield the same $L$-packet.  Consider the following archetypical example.
Let $K = \QQ_3$, $G = \SL_2$ and $T$ be an unramified anisotropic torus in $G$.  There are four depth zero
characters: two admissible and two inadmissible, notions defined below.  Since the two admissible characters are interchanged
by the action of the Weyl group, the corresponding $L$-packets are isomorphic \cite[\S10]{murnaghan:11}.
In this section we investigate depth zero characters of $T(K)$ that leave the association $\chi \mapsto \Lpack(\chi)$ invariant upon translation:
$$\Lpack(\chi) = \Lpack(\alpha\cdot\chi) \mbox{ for all depth zero admissible $\chi$}.$$

\begin{definition} \label{def:admissible}
Let $T$ be a $K$-minisotropic torus, that splits over an unramified
extension $L$ (see \cite[\S3]{reeder:08a}).  Suppose $\xi$ is a character of $T(L)_{\Gamma}$.
\begin{enumerate}
\item The pair $(T, \xi)$ is called \emph{admissible} if $\xi$ is not fixed
by any nontrivial element of $W(K)$ (c.f. Lemma \ref{lem:weyl_groups}); we
denote by $P_G(K)$ the set of admissible pairs in $G$.
\item We call two admissible pairs $(T, \xi)$ and $(T', \xi')$ \emph{isomorphic} if there
exists a $g \in G(K)$ such that $gT(K)g^{-1} = T'(K)$ and $\xi(t) = \xi'(gtg^{-1})$
for all $t \in T(K)$.
\end{enumerate}
Similarly, we will call a character of $T(K)$ \emph{admissible} if
it is not fixed by any nontrivial element of $W(K)$
(c.f. \cite[p. 802]{reeder-debacker:09a} and \cite[\S3]{reeder:08a})
\end{definition}

Note that this definition of admissible pair generalizes
Bushnell-Henniart's notion of admissible pair \cite{bushnell-henniart:10a} in
the case of unramified tori.  Indeed,
if $G = \GL_n$, and $T$ is an elliptic torus in $G$ splitting over
an unramified extension $L/K$, then one can show that
$W(K) = \Gamma$.  In this case, the following are equivalent conditions
on a character $\xi$ of $T(K) = \Lx$:
\begin{enumerate}
\item $\xi$ is fixed by a nontrivial element of $W(K)$,
\item $\xi$ is fixed by a nontrivial subgroup of $\Gamma$,
\item $\xi$ factors through the norm map $\Nm_{L/M}$ for some intermediate field $K \subseteq M \subset L$.
\end{enumerate}

Note that for non-adjoint groups it is not sufficient to consider only reflections.
For example, the depth zero character of the split torus in $\SL_3(\QQ_7)$ inflated from
$$\begin{pmatrix} 3^x & & \\ & 3^y & \\ & & 3^{-x-y} \end{pmatrix} \mapsto \zeta_6^{2x + 4y}$$
is fixed by a 3-cycle in the Weyl group and thus not admissible.

In the next section we will be particularly interested in depth zero characters; write $\hatT$ for the set of
depth zero characters of $T(\OK)$, $\Thadm$ for the admissible
ones and $\Thinadm$ for the inadmissible ones.  Each of these
sets is finite since they may be identified with characters of $T(k)$.

\begin{definition}
Write $Q_T$ for the set of $\alpha \in \hatT$ with the following property:
\begin{itemize}
\item For every $\chi \in \Thadm$ there is a $w \in W(K)$ with $\alpha = \frac{\chi}{w(\chi)}$.
\end{itemize}
\end{definition}

The $\SL_2(\QQ_3)$ example above has $Q_T$ of order two, but $Q_T$ is trivial for most tori.
We spend the rest of this section giving criteria constraining $Q_T$.

\begin{proposition} \label{irr-sub}
The set $Q_T$ is a subgroup of $\hatT$, contained within $\Thinadm$ and stable under the action of $W(K)$.
\end{proposition}
\begin{proof}
If $\alpha \in \Thadm \cap Q_T$ then there is some $w \in W(K)$ with
$\frac{\alpha}{w(\alpha)} = \alpha$, so $\alpha = 1$ which is not admissible.

We now show that $Q_T$ is a group.  Certainly $1 \in Q_T$.  Suppose
$\alpha, \alpha' \in Q_T$ and $\chi \in \Thadm$.  Then there are $w, w' \in W(K)$ with
\begin{align*}
\frac{\chi}{w(\chi)} &= \alpha, \\
\frac{w(\chi)}{w'(w(\chi))} &= \alpha'.
\end{align*}
Multiplying the two relations yields $\frac{\chi}{w'w(\chi)} = \alpha\alpha'$, so
$\alpha\alpha' \in Q_T$.  We finish by noting that $Q_T$ is finite and thus closure under
multiplication implies closure under inversion.

Finally, suppose $\tau \in W(K)$.  Given $\chi \in \Thadm$ with $\alpha = \frac{\chi}{w(\chi)}$ we have
$$\tau(\alpha) = \frac{\tau(\chi)}{\tau w(\chi)} = \frac{\tau(\chi)}{w' \tau(\chi)}$$
for some $w' \in W(K)$.  Since $\tau$ permutes the admissible characters we get that $\tau(\alpha) \in Q_T$.
\end{proof}

The condition on $\alpha \in Q_T$ is an extremely stringent one, and an abundance of admissible
characters will preclude a nontrivial $\alpha$.  We can make this statement precise:

\begin{proposition} \label{pigeonhole}
Suppose $\#\Thadm > (\# W(K) - 1) \cdot \# \Thinadm$.  Then $Q_T = \{ 1 \}$.
\end{proposition}
\begin{proof}
For $w \in W(K)$, set
$$S_w = \{\chi \in \Thadm \st  \frac{\chi}{w(\chi)} = \alpha\}.$$
Note that if $S_1$ is nonempty then we get $\alpha = 1$ immediately, so we may
assume the contrary.  Then by the pigeonhole principle, there is a $w \in W(K)$
with $\# S_w > \# \Thinadm$.  Pick $\chi \in S_w$; since $\# S_w > \#\Thinadm$
there is some $\chi' \in S_w$ with $\frac{\chi}{\chi'}$ admissible.  We now have
$$\frac{\chi}{w(\chi)} = \alpha = \frac{\chi'}{w(\chi')}$$
and therefore $\frac{\chi}{\chi'}$ is fixed by $w$.  Since $\frac{\chi}{\chi'}$ is admissible, we must have $w = 1$ and thus
$$\alpha = \frac{\chi}{\chi} = 1.$$
\end{proof}

Recall that Frobenius acts on $X^*(T)$ via an endomorphism $F = qF_0$, where
$F_0$ is an automorphism of finite order \cite[p. 82]{carter:93a}.  So it makes sense
to vary $q$: we fix $F_0$ and consider the tori dual to the $\Gal(\Fqb/\Fq)$-modules
with Frobenius acting through $qF_0$.

\begin{corollary}[{c.f. \cite[Lemma 8.4.2]{carter:93a}}]
Consider the family of tori $T_q$ with the same $F_0$.  Then for sufficiently large $q$,
$Q_{T_q} = \{ 1 \}$ (regardless of the $G$ in which $T_q$ is embedded).
\end{corollary}
\begin{proof}
We will write $T$ for a general torus in the family and $r$ for the common dimension.
Note that $\hatT$ is the set of $\Fq$ points of a dual torus, also of rank $r$ over $\Fq$.
For $w \in W(K)$ with $w \ne 1$ the centralizer $\Z_{\hatT}(w)$ is a proper $F$-stable
subgroup of $\hatT$, and thus $\dim(\Z_{\hatT}(w)) \le r - 1$.  By \cite[3.3.5]{carter:93a},
$\# \hatT$ is a polynomial in $q$ of degree $r$ and $\# \Z_{\hatT}(w)$ is a polynomial
in $q$ of degree at most $r-1$.  Thus the ratio
$$\frac{\# \Thadm}{\# \Thinadm} = \frac{\# \hatT - \sum_{1 \ne w \in W} \# \Z_{\hatT}(w)}{\sum_{1 \ne w \in W} \# \Z_{\hatT}(w)}$$
grows without bound as $q$ does.  There are finitely many possibilities for the absolute
Weyl group of $T$, so Proposition \ref{pigeonhole} gives the desired result.
\end{proof}

In computing $Q_T$ for small $q$ the following result is useful:

\begin{proposition} \label{orderdiv}
If $\alpha \in Q_T$ has order $d$ and $\chi \in \Thadm$ has order $m$ then $d$ divides $m$.
\end{proposition}
\begin{proof}
There is a $w \in W(K)$ with
$$\frac{\chi}{w(\chi)} = \alpha.$$
Since $w(\chi)$ also has order $m$, raising both sides to the $m$th power  yields $\alpha^m = 1$.
\end{proof}

Finally, we note that Lemma 8 of Bushnell-Henniart \cite[p. 511]{bushnell-henniart:10a} is equivalent to
the statement that $Q_T$ is trivial when $T$ is a $K$-minisotropic torus in $\GL_n$.

\section{Rectifiers for general reductive groups} \label{section:general_rectifiers}

Suppose that $G$ is a connected reductive group defined over a
$p$-adic field $K$.  Fix an unramified $K$-torus $T \subset G$ with splitting field $L$.
Let $\varphi : \Weil_K \rightarrow {}^L G$ be a
Langlands parameter for $G(K)$, and suppose that $\varphi$ factors
through a group of type L for $T$.  Any Langlands parameter with image in the normalizer
of a maximal torus will factor in this way for some $T$.

As in \S\ref{section:groups_of_type_L}, one can canonically
associate to $\varphi$ a character $\xi_{\varphi}$ of $T(L)_{\Gamma}$.
Recall again the Tate cohomology sequence
$$1 \rightarrow \HT{-1}(\Gamma,T(L)) \rightarrow T(L)_{\Gamma} \xrightarrow{\Nm} T(K)
= T(L)^{\Gamma} \rightarrow \HT{0}(\Gamma,T(L)) \rightarrow 1.$$
Suppose that $\HT{0}(\Gamma, T(L)) = 0$, in which case
$T(L)_{\Gamma}$ surjects onto $T(K)$.  Let us also suppose that
$\varphi$ does not factor through a proper Levi subgroup, so that the
representations in the $L$-packet associated to $\varphi$ are
conjecturally all supercuspidal (see \cite[\S 3.5]{reeder-debacker:09a}).
When $G = \GL_n$ we show in \S\ref{section:BH_compat} that 
$\HT{0}(\Gamma, T(L)) = \HT{-1}(\Gamma, T(L)) = 0$ and thus
$T(L)_{\Gamma} \cong T(K) \cong \Lx$.  In this case
$(L/K, \xi_{\varphi})$ is an admissible pair; to construct the local Langlands
correspondence one proceeds as in \S\ref{section:BH_recall} by
attaching the supercuspidal representation $\pi_{\xi_{\varphi} \cdot
  \mu_{\xi_{\varphi}}}$ to $\xi_{\varphi}$, via the construction of Bushnell and Henniart.

For other groups $G$ there are some constructions of supercuspidal $L$-packets $\Lpack(\chi)$
from characters $\chi$ of $T(K)$ \cite{reeder-debacker:09a, kaletha:13a, reeder:08a}.
However, as we have seen, a Langlands parameter $\varphi$ does not naturally
provide a character of $T(K)$, but rather a character of
$T(L)_{\Gamma}$.

\begin{definition} \label{def:rectifier}
  Let $T$ be a $K$-minisotropic torus in $G$, that splits over an unramified
  extension $L$.  A \emph{rectifier} for $T$ is a function
  $$\bmu : (T, \xi) \mapsto \mu_{\xi}$$
  which attaches to each $(T, \xi) \in P_G(K)$ a character
  $\mu_{\xi}$ of $T(L)_{\Gamma}$ satisfying the following conditions:

\begin{enumerate}
\item The character $\mu_{\xi}$ is tamely ramified (i.e. trivial on
  $T(\PL)_{\Gamma}$),

\item The character $\xi \cdot \mu_{\xi}$ descends to $T(K)$, is admissible,
and $\varphi \mapsto \Lpack(\xi_{\varphi} \cdot \mu_{\xi_{\varphi}})$
  is the local Langlands correspondence,

\item If $(T, \xi_1)$ and $(T, \xi_2)$ are admissible pairs such that
$\xi_1^{-1} \xi_2$ is tamely ramified then
$\mu_{\xi_1} = \mu_{\xi_2}$.
\end{enumerate}
We say that two rectifiers $\bmu$ and $\bmu'$ for $T$ are \emph{equivalent}
if there is some $\alpha \in Q_T$ so that
\begin{align*}
\mu_\xi' &= \alpha \mu_\xi \qquad \mbox{for depth zero $\xi$,} \\
\mu_\xi' &= \mu_\xi \qquad\ \  \mbox{for positive depth $\xi$.}
\end{align*}

\end{definition}

Since we have assumed $\HT{0}(\Gamma,T(L)) = 0$, the condition that $\xi \cdot \mu_\xi$
descends to $T(K)$ is equivalent to $\xi \cdot \mu_\xi$ vanishing on $\HT{-1}(\Gamma, T(L))$.  The
notion of equivalence is tailored for Theorem \ref{thm:unique_semisimple}; for some tori (such as the
$\SL_2(\mathbb{Q}_3)$ example at the beginning of \S \ref{Q_T}) there are
multiple equivalent rectifiers.

\begin{conjecture} \label{conj:unique_rectifier}
For $T$ as in Definition \ref{def:rectifier}, $T$ admits a unique rectifier up to equivalence.
\end{conjecture}

We note that, as the local Langlands correspondence is not known in general, we must restrict
ourselves to cases where supercuspidal $L$-packets have been constructed.
Since we are in the present paper considering the situation when $T$ is unramified,
we consider those $L$-packets constructed in \cite{reeder-debacker:09a} and \cite{reeder:08a}.
In the setting of Reeder \cite{reeder:08a},
we must further restrict our scope since his constructions do not apply to all admissible pairs.

\begin{definition}\label{def:general_pair}
Suppose $(T, \xi) \in P_G(K)$.
\begin{enumerate}
\item The \emph{depth} of $(T, \xi)$ is the integer $r$ so that $\xi$
is trivial on $T(\PL^{r+1})_{\Gamma}$ but nontrivial on
$T(\PL^{r})_{\Gamma}$
\item An admissible pair of depth $r$ is \emph{minimal}
if $\xi|_{T(\PL^{r})_{\Gamma}}$
is not fixed by any element of $W(K)$.
We denote by $\Pmin(K)$ the set
of minimal admissible pairs in $G$.
\item A \emph{weak rectifier} for $T \subset G$ is a function
\begin{align*}
\mumin : (T, \xi) \mapsto \mu_{\xi}
\end{align*}
which attaches to each $(T, \xi) \in \Pmin(K)$ a character
  $\mu_{\xi}$ of $T(L)_{\Gamma}$, satisfying conditions (1)-(3)
  of Definition \ref{def:rectifier}.
\end{enumerate}
We define equivalence of weak rectifiers as in Definition \ref{def:rectifier}.
\end{definition}

We note that this definition of minimal admissible pair generalizes
the definition of minimal admissible pair of Bushnell and Henniart in
the case of unramified tori (see \cite[\S2.2]{bushnell-henniart:05a}).

\begin{theorem} \label{thm:unique_semisimple}
For $G$ semisimple and $T$ as in Definition \ref{def:rectifier}, $T$ admits a unique weak rectifier up to equivalence.
\end{theorem}

\begin{proof}
We first prove existence.
First recall that $T$ can be defined
via Galois twisting by a Weyl group element $w$.  We defined in
\S\ref{section:gross_debacker_reeder}
a Langlands parameter $\varphi_{\hat{w}} :\Weil _K \rightarrow {}^L G$ by
sending Frobenius to the canonical lift
$\tilde{w} \in \widetilde{W}$ of $\hat{w} \in \hat{W}$, and by setting
$\varphi_{\hat{w}}$ to be trivial on $I_K$.
For semisimple $G$
we proved in Proposition \ref{prop:existenceofrectifier}
that the function $$(T, \xi) \mapsto \xi_{\varphi_{\hat{w}}}^{-1}$$
satisfies condition (2) of Definition \ref{def:rectifier}.  Moreover, the function also
satisfies condition (1): $\varphi_{\hat{w}}|_{I_K} \equiv 1$ and thus
$\xi_{\varphi_{\hat{w}}}^{-1}$ is unramified.
Finally, $\xi_{\varphi_{\hat{w}}}$ is independent of $\xi$
and thus condition (3) is automatically satisfied.
We may therefore set
$\mumin(T,\xi) = \xi_{\varphi_{\hat{w}}}^{-1}$.

We now prove uniqueness.
Let $\xi$ range over the set of characters of $T(L)_{\Gamma}$
such that $(T, \xi) \in \Pmin(K)$, and let
$\bmu$ and $\bmu'$ be weak rectifiers for
$T \subset G$.  By hypothesis, we have
$$\Lpack(\mu_{\xi} \cdot \xi) = \Lpack(\mu'_{\xi} \cdot \xi).$$
By \cite[\S10]{murnaghan:11}, there exists $w_{\xi} \in W(K)$,
depending on $\xi$, such that
$${}^{w_{\xi}} (\mu_\xi \cdot \xi) = \mu'_\xi \cdot \xi.$$
Suppose that $\xi$ has positive depth.
Restricting the equation
${}^{w_\xi} (\mu_\xi \cdot \xi) = \mu'_\xi \cdot \xi$
to $T(\PL)_{\Gamma}$, we get that ${}^{w_{\xi}} (\xi) = \xi$,
by condition (1) of Definition \ref{def:rectifier}.
Since $\xi$ is minimal, we get that $w_{\xi} = 1$, which implies
that $\mu_{\xi} = \mu'_{\xi}$.

Now suppose that $\xi$ has depth zero.
Define $\lambda$ on $T(\OL)_{\Gamma} \cong T(\OK)$ by $\lambda = ({}^{w_\xi} (\mu_\xi))^{-1} \cdot \mu'_\xi$,
which is independent of $\xi$ by condition (3). The equation ${}^{w_\xi} (\mu_\xi \cdot \xi) = \mu'_\xi \cdot \xi$
implies that $\lambda \in Q_T$.  Since $\mu_\xi \cdot \xi$ and $\mu_\xi' \cdot \xi$ descend to
$T(K)$ by condition (2) of Definition \ref{def:rectifier}, $\mu_\xi$ and $\mu_\xi'$ have the
same restriction to $\hat{H}^{-1}(\Gamma, T(L))$.  Since $G$ is semisimple we may pull $\lambda$ back to
a character on $T(L)_\Gamma$, vanishing on $\HT{-1}(\Gamma, T(L))$.  We get that $\mu'_\xi = \lambda\mu_\xi$
and thus $\bmu$ is equivalent to $\bmu'$.
\end{proof}

\begin{remark} $ $
\begin{enumerate}
\item The condition $\HT{0}(\Gamma, T(L)) = 0$ was necessary in order to obtain a character on $T(K)$ rather
than the image of the norm map $T(L) \mapsto T(K)$.  For non-semisimple groups where $\HT{0}(\Gamma, T(L))$
is nontrivial we hope that the recipe for the central character in \cite{gross-reeder:09a} will provide an extension to all of $T(K)$.
\item The rectifier in our setting is constant as a function of $\xi$.  We expect a dependence on $\xi$ for ramified tori.
\item The behavior of rectifiers under change of group is not yet clear to us.  There may be a natural relationship between
rectifiers when a torus is embedded into two
different reductive groups with isomorphic Weyl groups.  Similarly, when given an embedding
$H \subset G$, a natural relationship between the rectifiers for tori in $H$ and $G$ would allow us to apply the results of \cite{bushnell-henniart:10a} to rectifiers for general groups.
\end{enumerate}
\end{remark}

\section{Compatibility with Bushnell-Henniart} \label{section:BH_compat}

In this section we show that our function $\mumin$
agrees with the rectifier of Bushnell-Henniart in the depth
zero setting: see Theorem \ref{thm:bh_agreement}.
Let $L = K_n$ and set 
$T = \Res_{L/K}(\Gm)$.  We begin by computing the Tate cohomology groups of $T$.

\begin{proposition}
$\HT{0}(\Gamma, X_*(T)) = 0$.
\end{proposition}

\begin{proof}
Since $\Gamma$ acts on $X_*(T)$ by permuting basis vectors,
$X_*(T)^{\Gamma}$ is the copy of $\ZZ$ embedded diagonally in
$X_*(T) = \mathbb{Z}^n$.  Note that $$\Nm(1,0,0,\cdots,0) = (1,1,\cdots,1),$$ so
$X_*(T)^{\Gamma} \subset \Nm(X_*(T))$.
\end{proof}

\begin{proposition}
$\HT{-1}(\Gamma, X_*(T)) = 0$.
\end{proposition}

\begin{proof}
We note that $(a_1, a_2, \cdots, a_n) \in \ker(\Nm)$ if and only if $\sum_{i=1}^n a_i = 0$.
It is then easy to see that $\ker(\Nm)$ is generated by $e_i - e_j$ for $i < j$, where
$e_i$ are the standard basis of $\mathbb{Z}^n$.  But $e_i - e_j = (1 - \tau)e_i$ for some
$\tau \in \Gamma$, since $\Gamma$ acts by cyclic shift.  Thus $\ker(\Nm) \subset I_{\Gamma}(X_*(T))$.
\end{proof}

The Tate cohomology exact sequence for $T$ therefore reduces to
$$1 \rightarrow T(L)_{\Gamma} \xrightarrow{\sim} T(K) \rightarrow 1$$ by
Corollary \ref{cor:cohom_tori}.  We now need a basic result about powers of lifts
of Coxeter elements in $\GL_{n}(\CC)$.

\begin{proposition}\label{prop:powers_of_lifts}
Let $\hat{w}$ be a Coxeter element of $\GL_{n}(\CC)$, and let $\tilde{w}$ be the
canonical lift of $\hat{w}$ to $\widetilde{W}$. Then $\tilde{w}^n = (-1)^{n-1}$ as
as scalar matrix in $\GL_{n}(\CC)$.
\end{proposition}

\begin{proof}
See \cite[\S3.1]{zaremsky:ppa}.
\end{proof}

We can now describe the image of $\mumin$ in the setting of
depth zero supercuspidal representations of $\GL_{n}(K)$.  Write
$\varphi$ for $\varphi_{\hat{w}}$ (see Definition \ref{def:phiu}) and $\mu$ for $\xi_{\varphi}^{-1}$ .

\begin{proposition} \label{prop:rectifier_agreement}
$\mu$ is unramified and
$\mu(\varpi) = (-1)^{n-1}$.
\end{proposition}

\begin{proof}
Let $\sigma$ generate $\Gal(L/K)$.  Then $T(L) \cong \Lx \times \Lx \times \cdots \times \Lx$ and
$$T(K) = \{(x, \sigma(x), \sigma^2(x), \cdots, \sigma^{n-1}(x)) : x \in \Lx \} \cong \Lx.$$
A uniformizer $\varpi$ in $\Kx \subset \Lx$
therefore corresponds to $(\varpi, \varpi, \cdots, \varpi) \in T(K)$, whose
preimage under $\Nm$ is the class of $(\varpi, 1, 1, \cdots, 1)$ in $T(L)_{\Gamma}$.
By \cite[\S 2.4]{serre:LocalClassFieldThy}, $\varpi$ corresponds to $\Fr^n$ under the Artin
reciprocity map for $L$.  Now by Proposition \ref{prop:powers_of_lifts}
and the local Langlands correspondence for tori we get
$\mu(\varpi) = (-1)^{n-1}$.
Finally, $\varphi|_{I_K} \equiv 1$ implies that $\mu$ is
unramified.
\end{proof}

\begin{theorem} \label{thm:bh_agreement}
  If $G = \GL_{n}(K)$ and fixed $T$, the constant function $(T,\xi) \mapsto \mu$ agrees with
  the rectifier of Bushnell-Henniart for depth zero $\xi$.
\end{theorem}

\begin{proof}
This result follows from Proposition \ref{prop:rectifier_agreement} and Proposition \ref{prop:BH_result1}.
\end{proof}

We end this section by explaining why the Tits group lift $\tilde{w}$ is forced upon us.
Suppose we define
$\varphi' : \Weil_K \rightarrow \GL_{n}(\CC)$ by $\varphi'|_{I_K} \equiv 1$ and
$\varphi'(\Fr)$ to be a lift of an elliptic element $\hat{w}$ in $\hat{W}$.
Then \cite[p. 824]{reeder-debacker:09a} and \cite[\S6]{reeder:08a} imply that the characteristic
polynomial of $\varphi'(\Fr)$ is $X^n - a$, for some $a \in \CCx$.  One can see that,
by arguments analogous to those in Proposition \ref{prop:rectifier_agreement},
$\xi_{\varphi'}(\varpi) = a$.  By Proposition \ref{prop:BH_result1}, we are
forced to set $a = (-1)^{n-1}$.  Finally, one can show by an inductive argument that the
canonical lift $\tilde{w}$ of $\hat{w}$ to $\widetilde{W}$ has characteristic polynomial $X^n - (-1)^{n-1}$,
so that $\varphi'(\Fr)$ is indeed the canonical lift of $\hat{w}$ to $\widetilde{W}$ up to conjugacy.

\bibliographystyle{amsalpha}
\bibliography{Biblio}

\providecommand{\bysame}{\leavevmode\hbox to3em{\hrulefill}\thinspace}
\providecommand{\MR}{\relax\ifhmode\unskip\space\fi MR }
\providecommand{\MRhref}[2]{%
  \href{http://www.ams.org/mathscinet-getitem?mr=#1}{#2}
}
\providecommand{\href}[2]{#2}
\begin{thebibliography}{Hum75}

\bibitem[Adl98]{adler:98a}
Jeffrey Adler, \emph{{Refined anisotropic $K$-types and supercuspidal
  representations}}, Pacific J. Math. (1998), 1--32.

\bibitem[Adr13]{adrian:13a}
Moshe Adrian, \emph{{A new realization of the Langlands correspondence for
  $PGL(2,F)$}}, Journal of Number Theory \textbf{133} (2013), 446--474.

\bibitem[AL12]{adrian-lansky:ppa}
Moshe Adrian and Joshua Lansky, \emph{{A real groups construction of the tame
  local Langlands correspondence for $PGSp(4,F)$}}, Preprint, arXiv:1209.6045,
  2012.

\bibitem[AW67]{atiyah-wall:CohomologyGrps}
Michael Atiyah and Charles Wall, \emph{Cohomology of groups}, {Algebraic number
  theory} (John W.~S. Cassels and Albrecht Frohlich, eds.), Academic Press,
  London, 1967, pp.~84--115.

\bibitem[BH05]{bushnell-henniart:05a}
Colin Bushnell and Guy Henniart, \emph{{The essentially tame local Langlands
  correspondence, II: totally ramified representations}}, Compositio Math.
  \textbf{141} (2005), 979--1011.

\bibitem[BH06]{bushnell-henniart:06a}
\bysame, \emph{{The local Langlands conjecture for $GL(2)$}}, Springer-Verlag,
  Berlin, 2006.

\bibitem[BH10]{bushnell-henniart:10a}
\bysame, \emph{{The essentially tame local Langlands correspondence, III: the
  general case}}, Proc. Lond. Math. Soc. \textbf{101} (2010), no.~2, 497--553.

\bibitem[BK93]{bushnell-kutzko:AdmissibleDual}
Colin Bushnell and Phil Kutzko, \emph{{The admissible dual of $GL(n)$ via
  compact open subgroups}}, Ann. Math. Studies 129, Princeton U.P., 1993.

\bibitem[BLR80]{bosch-lutkebohmert-reynaud:NeronModels}
Siegfried Bosch, Werner L\"utkebohmert, and Michel Reynaud, \emph{{N\'eron
  models}}, Springer-Verlag, Berlin, 1980.

\bibitem[Bor79]{borel:79a}
Armand Borel, \emph{{Automorphic L-functions}}, {Automorphic forms,
  representations and L-functions} (Providence R.I.), Proceedings of Symposia
  in Pure Mathematics, no. 33, part 2, American Math Society, 1979, pp.~27--61.

\bibitem[Car93]{carter:93a}
Roger~W. Carter, \emph{{Finite groups of Lie type: conjugacy classes and
  complex characters}}, Wiley \& Sons, Chichester, 1993.

\bibitem[DR09]{reeder-debacker:09a}
Stephen DeBacker and Mark Reeder, \emph{{Depth-zero supercuspidal L-packets and
  their stability}}, Annals of Math \textbf{169} (2009), no.~3, 795--901.

\bibitem[GR10]{gross-reeder:09a}
Benedict~H. Gross and Mark Reeder, \emph{{Arithmetic invariants of discrete
  Langlands parameters}}, Duke Math Journal (2010), 431--508.

\bibitem[How97]{howe:77a}
Roger Howe, \emph{{Tamely ramified supercuspidal representations of
  $GL_n(F)$}}, Pacific J. Math. \textbf{73} (1997), 437--460.

\bibitem[Hum75]{humphreys:LinAlgGrps}
James~E. Humphreys, \emph{{Linear algebraic groups}}, Graduate Texts in
  Mathematics 21, Springer-Verlag, 1975.

\bibitem[Kal13]{kaletha:13a}
Tasho Kaletha, \emph{{Simple wild $L$-packets}}, J. Inst. Math. Jussieu
  \textbf{12} (2013), no.~1, 43--75.

\bibitem[Lan56]{lang:56a}
Serge Lang, \emph{Algebraic groups over finite fields}, Amer. J. Math.
  \textbf{78} (1956), 555--563.

\bibitem[MP96]{moy-prasad:96a}
Allen Moy and Gopal Prasad, \emph{{Jacquet functors and unrefined minimal
  $K$-types}}, Comment. Math. Helv. \textbf{71} (1996), no.~1, 98--121.

\bibitem[Mur11]{murnaghan:11}
Fiona Murnaghan, \emph{Parametrization of tame supercuspidal representations},
  On certain $L$-functions: a volume in honor of Freydoon Shahidi on the
  occasion of his 60th birthday (James Arthur, James~W. Cogdell, Stephen
  Gelbart, David Goldberg, and Dinakar Ramakrishnan, eds.), Clay Math Proc. 13,
  2011, pp.~439--470.

\bibitem[Ree08]{reeder:08a}
Mark Reeder, \emph{{Supercuspidal $L$-packets of positive depth and twisted
  Coxeter elements}}, J. Reine Angew. Math. \textbf{620} (2008), 1--33.

\bibitem[Ree09]{reeder:09a}
Mark Reeder, \emph{Torsion automorphisms of simple lie algebras}, To appear in
  L'Enseignement Mathematique, 2009.

\bibitem[Ser67]{serre:LocalClassFieldThy}
Jean-Pierre Serre, \emph{Local class field theory}, {Algebraic number theory}
  (John W.~S. Cassels and Albrecht Frohlich, eds.), Academic Press, London,
  1967, pp.~129--162.

\bibitem[Ser88]{serre:AlgGrpsClassFields}
\bysame, \emph{{Algebraic groups and class fields}}, Graduate Texts in
  Mathematics 117, Springer-Verlag, New York, 1988.

\bibitem[Tam12]{tam:12a}
Geo Kam-Fai Tam, \emph{Transfer relations in essentially tame local langlands
  correspondence}, Ph.D. thesis, University of Toronto, 2012.

\bibitem[Tit66]{tits:66a}
Jacques Tits, \emph{Normalisateurs de tores}, J. Algebra \textbf{4} (1966),
  96--116.

\bibitem[Xar93]{xarles:93a}
Xavier Xarles, \emph{{The scheme of connected components of the N\'eron model
  of an algebraic torus}}, J. Reine Angew. Math. \textbf{437} (1993), 167--180.

\bibitem[Yu03]{yu:03a}
Jiu-Kang Yu, \emph{{Smooth models associated to concave functions in
  Bruhat-Tits theory}}, Preprint, 2003.

\bibitem[Zar11]{zaremsky:ppa}
Matthew Zaremsky, \emph{{Representatives of elliptic Weyl group elements in
  algebraic groups}}, Preprint, arXiv:1109.5487, 2011.

\end{thebibliography}

\end{document}